\documentclass[twoside]{amsart}

\newif\ifdraft
\draftfalse
\usepackage{amsmath,amsfonts,amsthm,mathrsfs}
\usepackage{amssymb}

\usepackage[ocgcolorlinks,unicode,bookmarks]{hyperref}
\usepackage[usenames,dvipsnames]{xcolor}
\hypersetup{colorlinks=true,citecolor=NavyBlue,linkcolor=Brown,urlcolor=Green} 

\usepackage{expl3}

\ExplSyntaxOn
\prop_new:N \g_cite_map_prop
\tl_new:N \l_citekey_result_tl

\cs_new:Npn \mapcitekey #1#2 {
  \clist_map_inline:nn {#2}
       {  \prop_gput:Nnn  \g_cite_map_prop  {##1} {#1}   }
}

\cs_new:Npn \getcitekey #1 {
   \prop_get:NoN \g_cite_map_prop{#1}  \l_citekey_result_tl
   \quark_if_no_value:NF \l_citekey_result_tl
       {  \tl_set_eq:NN #1  \l_citekey_result_tl  }
}

\cs_new:Npn \showcitekeymaps {\prop_show:N  \g_cite_map_prop }
\ExplSyntaxOff

\usepackage{etoolbox}
\makeatletter
\patchcmd{\@citex}{\if@filesw}{\getcitekey\@citeb \if@filesw}%
    {\typeout{*** SUCCESS ***}}{\typeout{*** FAIL ***}}
\patchcmd{\nocite}{\if@filesw}{\getcitekey\@citeb \if@filesw}%
    {\typeout{*** SUCCESS ***}}{\typeout{*** FAIL ***}}
\makeatother

\usepackage{enumitem}

\usepackage{chngcntr}

\ifdraft
\usepackage[notcite,notref,color]{showkeys}

\definecolor{labelkey}{gray}{0.5}
\fi

\usepackage{tikz}
\usepackage{tikz-cd}

\newcommand{\Dmod}{\mathscr{D}}
\newcommand{\Mmod}{\mathcal{M}}

\newcommand{\ZZ}{\mathbb{Z}}
\newcommand{\QQ}{\mathbb{Q}}

\newcommand{\CC}{\mathbb{C}}

\newcommand{\PP}{\mathbb{P}}

\DeclareMathOperator{\gr}{gr}

\DeclareMathOperator{\Pic}{Pic}

\newcommand{\shf}[1]{\mathscr{#1}}

\def\overbar#1#2#3{{%
	\setbox0=\hbox{\displaystyle{#1}}%
	\dimen0=\wd0
	\advance\dimen0 by -#2 
	\vbox {\nointerlineskip \moveright #3 \vbox{\hrule height 0.3pt width \dimen0}%
		\nointerlineskip \vskip 1.5pt \box0}%
}}

\newcommand{\tl}{t_{\ast}}

\newcommand{\DD}{\mathbb{D}}

\newcommand{\shO}{\shf{O}}

\makeatletter
\let\@@seccntformat\@seccntformat
\renewcommand*{\@seccntformat}[1]{%
  \expandafter\ifx\csname @seccntformat@#1\endcsname\relax
    \expandafter\@@seccntformat
  \else
    \expandafter
      \csname @seccntformat@#1\expandafter\endcsname
  \fi
    {#1}%
}
\newcommand*{\@seccntformat@subsection}[1]{%
  \textbf{\csname the#1\endcsname.}
}
\makeatother

\makeatletter
\let\@paragraph\paragraph
\renewcommand*{\paragraph}[1]{%
	\vspace{0.3\baselineskip}%
	\@paragraph{\textit{#1}}%
}
\makeatother

\counterwithin{equation}{subsection}
\counterwithout{subsection}{section}
\counterwithin{figure}{subsection}

\newtheorem{theorem}[equation]{Theorem}
\newtheorem*{theorem*}{Theorem}
\newtheorem{lemma}[equation]{Lemma}
\newtheorem*{lemma*}{Lemma}
\newtheorem{corollary}[equation]{Corollary}
\newtheorem*{corollary*}{Corollary}

\newtheorem*{proposition*}{Proposition}

\newtheorem*{conjecture*}{Conjecture}

\theoremstyle{definition}
\newtheorem{definition}[equation]{Definition}
\newtheorem*{definition*}{Definition}
\newtheorem{remark}[equation]{Remark}
\newtheorem*{remark*}{Remark}

\newtheorem{question}[equation]{Question}
\newtheorem*{question*}{Question}

\newtheorem{example}[equation]{Example}
\newtheorem*{example*}{Example}
\newtheorem*{problem*}{Problem}

\theoremstyle{plain}

\newcommand{\theoremref}[1]{\hyperref[#1]{Theorem~\ref*{#1}}}
\newcommand{\lemmaref}[1]{\hyperref[#1]{Lemma~\ref*{#1}}}
\newcommand{\definitionref}[1]{\hyperref[#1]{Definition~\ref*{#1}}}
\newcommand{\propositionref}[1]{\hyperref[#1]{Proposition~\ref*{#1}}}
\newcommand{\conjectureref}[1]{\hyperref[#1]{Conjecture~\ref*{#1}}}
\newcommand{\corollaryref}[1]{\hyperref[#1]{Corollary~\ref*{#1}}}
\newcommand{\exampleref}[1]{\hyperref[#1]{Example~\ref*{#1}}}
\newcommand{\exerciseref}[1]{\hyperref[#1]{Exercise~\ref*{#1}}}

\makeatletter
\let\old@caption\caption
\renewcommand*{\caption}[1]{%
	\setcounter{figure}{\value{equation}}%
	\stepcounter{equation}%
	\old@caption{#1}\relax%
}
\makeatother

\newcounter{intro}

\newtheorem{intro-conjecture}[intro]{Conjecture}
\newtheorem{intro-corollary}[intro]{Corollary}
\newtheorem{intro-theorem}[intro]{Theorem}

\newcommand{\shB}{\mathcal{B}}

\newcommand{\parref}[1]{\hyperref[#1]{\S\ref*{#1}}}
\newcommand{\chapref}[1]{\hyperref[#1]{Chapter~\ref*{#1}}}

\makeatletter
\newcommand*\if@single[3]{%
  \setbox0\hbox{${\mathaccent"0362{#1}}^H$}%
  \setbox2\hbox{${\mathaccent"0362{\kern0pt#1}}^H$}%
  \ifdim\ht0=\ht2 #3\else #2\fi
  }
\newcommand*\rel@kern[1]{\kern#1\dimexpr\macc@kerna}
\newcommand*\widebar[1]{\@ifnextchar^{{\wide@bar{#1}{0}}}{\wide@bar{#1}{1}}}
\newcommand*\wide@bar[2]{\if@single{#1}{\wide@bar@{#1}{#2}{1}}{\wide@bar@{#1}{#2}{2}}}
\newcommand*\wide@bar@[3]{%
  \begingroup
  \def\mathaccent##1##2{%
    \if#32 \let\macc@nucleus\first@char \fi
    \setbox\z@\hbox{$\macc@style{\macc@nucleus}_{}$}%
    \setbox\tw@\hbox{$\macc@style{\macc@nucleus}{}_{}$}%
    \dimen@\wd\tw@
    \advance\dimen@-\wd\z@
    \divide\dimen@ 3
    \@tempdima\wd\tw@
    \advance\@tempdima-\scriptspace
    \divide\@tempdima 10
    \advance\dimen@-\@tempdima
    \ifdim\dimen@>\z@ \dimen@0pt\fi
    \rel@kern{0.6}\kern-\dimen@
    \if#31
      \overline{\rel@kern{-0.6}\kern\dimen@\macc@nucleus\rel@kern{0.4}\kern\dimen@}%
      \advance\dimen@0.4\dimexpr\macc@kerna
      \let\final@kern#2%
      \ifdim\dimen@<\z@ \let\final@kern1\fi
      \if\final@kern1 \kern-\dimen@\fi
    \else
      \overline{\rel@kern{-0.6}\kern\dimen@#1}%
    \fi
  }%
  \macc@depth\@ne
  \let\math@bgroup\@empty \let\math@egroup\macc@set@skewchar
  \mathsurround\z@ \frozen@everymath{\mathgroup\macc@group\relax}%
  \macc@set@skewchar\relax
  \let\mathaccentV\macc@nested@a
  \if#31
    \macc@nested@a\relax111{#1}%
  \else
    \def\gobble@till@marker##1\endmarker{}%
    \futurelet\first@char\gobble@till@marker#1\endmarker
    \ifcat\noexpand\first@char A\else
      \def\first@char{}%
    \fi
    \macc@nested@a\relax111{\first@char}%
  \fi
  \endgroup
}
\makeatother

\newcommand{\I}{\mathcal{I}}
\renewcommand{\DD}{\mathbf{D}}
\newcommand{\AAA}{\mathbf{A}}

\mapcitekey{Campana+Peternell:GeometricStability}{CPet}
\mapcitekey{Campana+Paun:Quotients}{CP2}
\mapcitekey{Campana+Paun:OrbifoldOne}{CP1}
\mapcitekey{Kollar:SingularitiesPairs}{Kollar}
\mapcitekey{Kollar:DualizingI}{Kollar1}
\mapcitekey{Kollar:DualizingII}{Kollar3}
\mapcitekey{Kollar:Subadditivity}{Kollar2}
\mapcitekey{Kollar:Shafarevich}{Kollar4}
\mapcitekey{Viehweg:WP1}{Viehweg1}
\mapcitekey{Viehweg:WP2}{Viehweg2}
\mapcitekey{Viehweg:Survey}{Viehweg3}
\mapcitekey{Viehweg:Hilbert1}{Viehweg4}
\mapcitekey{Popa+Wu:WeakPositivity}{PW}
\mapcitekey{Viehweg+Zuo:Isotriviality}{VZ1}
\mapcitekey{Viehweg+Zuo:BaseSpaces}{VZ2}
\mapcitekey{Viehweg+Zuo:Brody}{VZ3}
\mapcitekey{Zuo:Kernels}{Zuo}
\mapcitekey{Popa+Schnell:OneForms}{PS}
\mapcitekey{Popa+Schnell:Hyperbolicity}{PS2}
\mapcitekey{Popa+Schnell:mhmgv}{mhmgv}
\mapcitekey{Saito:KC}{Saito-KC}
\mapcitekey{BDPP:Pseudoeffective}{BDPP}
\mapcitekey{Esnault+Viehweg:VanishingTheorems}{EV}
\mapcitekey{Saito:HodgeModules}{Saito-MHP}
\mapcitekey{Saito:MixedHodgeModules}{Saito-MHM}
\mapcitekey{Saito:b-function}{Saito-B}
\mapcitekey{Saito:Steenbrink}{Saito-Spectrum}
\mapcitekey{Saito:Kollar}{Saito-Kollar}
\mapcitekey{Saito:YPG}{Saito-YPG}
\mapcitekey{Saito:MLCT}{Saito-MLCT}
\mapcitekey{Saito:Microlocal}{Saito-MV}
\mapcitekey{Saito:HodgeFiltration}{Saito-HF}
\mapcitekey{Popa:SaitoVanishing}{Popa}
\mapcitekey{Popa:UtahSurvey}{Popa2}
\mapcitekey{Brunebarbe:Variations}{Brunebarbe}
\mapcitekey{Schnell:HolonomicAbelian}{Schnell}
\mapcitekey{Kebekus+Kovacs:FamiliesSurfaces}{KK1}
\mapcitekey{Kebekus+Kovacs:FamiliesCompact}{KK2}
\mapcitekey{Kebekus+Kovacs:FamiliesThreefolds}{KK3}
\mapcitekey{Kovacs:FamiliesSurvey}{Kovacs}
\mapcitekey{Kovacs:BaseNef}{Kovacs2}
\mapcitekey{Kovacs:Hyperbolicity}{Kovacs3}
\mapcitekey{Kovacs+Schwede:HTMMP}{KS}
\mapcitekey{Kebekus:VZSurvey}{Kebekus}
\mapcitekey{Patakfalvi:Hyperbolicity}{Patakfalvi}
\mapcitekey{Kawamata:AdditivityMMP}{Kawamata}
\mapcitekey{Mori:Survey}{Mori}
\mapcitekey{Schnell:sanya}{Schnell1}
\mapcitekey{Schnell:SaitoVanishing}{Schnell2}
\mapcitekey{Schnell:weakpos}{Schnell3}
\mapcitekey{Schnell:Neron}{Schnell4}
\mapcitekey{Schnell:ExtendedHodge}{Schnell5}
\mapcitekey{Taji:special}{Taji}
\mapcitekey{Mustata+Popa:HodgeIdeals}{MP1}
\mapcitekey{Mustata+Popa:Restriction}{MP2}
\mapcitekey{Mustata+Popa:HodgeIdealsII}{MP3}
\mapcitekey{Mustata+Popa:HodgeIdealsVfil}{MP4}
\mapcitekey{Mustata+Olano+Popa:Rational}{MOP}
\mapcitekey{Pareschi+Popa+Schnell:Kaehler}{PPS}
\mapcitekey{Esnault+Viehweg:VanishingTheorems}{EV}
\mapcitekey{Wu:Vanishing}{Wu}
\mapcitekey{Suh:Vanishing}{Suh}
\mapcitekey{Hotta+Takeuchi+Tanisaki:Book}{HTT}
\mapcitekey{Ein+Lazarsfeld:Theta}{EL}
\mapcitekey{Schmid:VHS}{Schmid}
\mapcitekey{Hacon+Kovacs:OneForms}{HK}
\mapcitekey{Luo+Zhang:Threefolds}{LZ}
\mapcitekey{Zhang:FormsAmple}{Zhang}
\mapcitekey{BCHM:MinimalModels}{BCHM}
\mapcitekey{Saito:Kaehler}{Saito-Kaehler}
\mapcitekey{Green+Lazarsfeld:GV1}{GL1}
\mapcitekey{Green+Lazarsfeld:GV2}{GL2}
\mapcitekey{Hacon:GV}{Hacon}
\mapcitekey{Chen+Jiang:Pushforward}{CJ}
\mapcitekey{Wang:TorsionPoints}{Wang}
\mapcitekey{Pareschi:Standard}{Pareschi}
\mapcitekey{Chen+Hacon:abelian}{CH1}
\mapcitekey{Jiang:effective}{Jiang}
\mapcitekey{Simpson:Subspaces}{Simpson}
\mapcitekey{Pareschi+Popa:GV}{PP3}
\mapcitekey{Pareschi+Popa:regI}{PP4}
\mapcitekey{Pareschi+Popa:regIII}{PP5}
\mapcitekey{Lazarsfeld:PositivityII}{Lazarsfeld}
\mapcitekey{Wu:Multiplier}{Wu2}
\mapcitekey{Budur+Saito:Multiplier}{BS}
\mapcitekey{Dimca+Maisonobe+Saito:Spectrum}{DMS}
\mapcitekey{Sabbah+Schnell:MHMProject}{SS}
\mapcitekey{Lichtin:Bfunction}{Lichtin}
\mapcitekey{Lichtin:UpperSemicontinuity}{Lichtin2}
\mapcitekey{Kashiwara:Bfunction}{Kashiwara}
\mapcitekey{Kashiwara:VanishingCycleSheaves}{Kashiwara2}
\mapcitekey{Malgrange:BernsteinSato}{Malgrange}
\mapcitekey{ELSV:JumpingCoefficients}{ELSV}
\mapcitekey{Yano:b-function}{Yano}
\mapcitekey{ELN}{ELN}
\mapcitekey{Demailly:VeryAmpleness}{Demailly}
\mapcitekey{Ran}{Ran}
\mapcitekey{Sabbah:Vfiltration}{Sabbah}
\mapcitekey{Granger:Survey}{Granger}
\mapcitekey{Mustata+Olano+Popa:Rational}{MOP}
\mapcitekey{Dutta:Toric}{Dutta}
\mapcitekey{Wang:TorsionPoints}{Wang}
\mapcitekey{Wei:LogForms}{Wei}
\mapcitekey{Siu:HyperbolicitySurvey}{Siu}
\mapcitekey{MSS}{MSS}

\begin{document}

\title[$\Dmod$-modules in birational geometry]{$\Dmod$-modules in birational geometry}

\author{Mihnea Popa}
\address{Department of Mathematics, Northwestern University,
2033 Sheridan Road, Evanston, IL 60208, USA} 
\email{mpopa@math.northwestern.edu}

\thanks{The author was partially supported by the NSF grant DMS-1700819.}


\begin{abstract}
It is well known that numerical quantities arising from the theory of $\Dmod$-modules are related to
invariants of singularities in birational geometry. This paper surveys a deeper relationship between the 
two areas, where the numerical connections are enhanced to sheaf theoretic constructions facilitated by the theory of mixed 
Hodge modules. The emphasis is placed on the recent theory of Hodge ideals.
\end{abstract}


\maketitle

\markboth{MIHNEA POPA} 
{$\Dmod$-modules in birational geometry}

\makeatletter
\newcommand\@dotsep{4.5}
\def\@tocline#1#2#3#4#5#6#7{\relax
  \ifnum #1>\c@tocdepth 
  \else
    \par \addpenalty\@secpenalty\addvspace{#2}%
    \begingroup \hyphenpenalty\@M
    \@ifempty{#4}{%
      \@tempdima\csname r@tocindent\number#1\endcsname\relax
    }{%
      \@tempdima#4\relax
    }%
    \parindent\z@ \leftskip#3\relax
    \advance\leftskip\@tempdima\relax
    \rightskip\@pnumwidth plus1em \parfillskip-\@pnumwidth
    #5\leavevmode\hskip-\@tempdima #6\relax
    \leaders\hbox{$\m@th
      \mkern \@dotsep mu\hbox{.}\mkern \@dotsep mu$}\hfill
    \hbox to\@pnumwidth{\@tocpagenum{#7}}\par
    \nobreak
    \endgroup
  \fi}
\def\l@section{\@tocline{1}{0pt}{1pc}{}{\bfseries}}
\def\l@subsection{\@tocline{2}{0pt}{25pt}{5pc}{}}
\makeatother



\setlength{\parskip}{0.5\baselineskip}

\subsection{Introduction}
Ad hoc arguments based on differentiating rational functions or sections of line bundles abound in complex and birational geometry. To pick just a couple of examples, topics where such arguments have made a deep impact are the study of adjoint linear series on smooth projective varieties, see for instance Demailly's work on effective very ampleness \cite{Demailly} and its more algebraic incarnation in \cite{ELN}, and the study of hyperbolicity, see for instance Siu's survey \cite{Siu} and the references therein.

A systematic approach, as well as an enlargement of the class of objects to which differentiation techniques apply, is provided by the theory of $\Dmod$-modules, which has however only recently begun to have a stronger impact in birational geometry. 
The new developments are mainly due to a better understanding of Morihiko Saito's theory of mixed Hodge modules \cite{Saito-MHP}, \cite{Saito-MHM}, and hence to deeper connections with Hodge theory and coherent sheaf theory. Placing problems in this context automatically brings in important tools such as vanishing theorems, perverse sheaves, or the $V$-filtration, in a unified way.

Connections between invariants arising from log resolutions of singularities and invariants arising from the theory of 
$\Dmod$-modules go back a while however. A well-known such instance is the fact that the log canonical threshold of a function 
$f$ on (say) $\CC^n$ coincides with the negative of the largest root of the Bernstein-Sato polynomial $b_f (s)$; see e.g. \cite{Yano}, \cite{Kollar}. Numerical data on log resolutions plays a role towards the study of other roots of the Bernstein-Sato polynomial \cite{Kashiwara}, \cite{Lichtin}, though our understanding of these is far less thorough. Going one step further, a direct relationship between the multiplier ideals of a hypersurface in a smooth variety $X$ and the $V$-filtration it induces on $\shO_X$  was established in \cite{BS}. 

After reviewing some of this material, in this paper I focus on one direction of further development, worked out jointly with Musta\c t\u a \cite{MP1}, \cite{MP2}, \cite{MP3}, \cite{MP4} as well as by Saito in \cite{Saito-MLCT}, namely the theory of what we call \emph{Hodge ideals}. This is a way of thinking about the Hodge filtration (in the sense of mixed Hodge modules) on the sheaf of functions with arbitrary poles along a hypersurface, or twists thereof, and is closely related to both the singularities of the hypersurface and the Hodge theory of its complement.  There are two key approaches that have proved useful towards understanding Hodge ideals:
\begin{enumerate}
\item A birational study in terms of log resolutions, modeled on the algebraic theory of multiplier ideals, which Hodge ideals generalize, \cite{MP1}, \cite{MP3}.
\item A comparison with the (microlocal) $V$-filtration, using its interaction with the Hodge filtration in the case of mixed Hodge modules, \cite{Saito-MLCT}, \cite{MP4}.
\end{enumerate}

Hodge ideals are indexed by the non-negative integers; at the $0$-th step, they essentially coincide with multiplier ideals. Beyond the material presented in this paper, by analogy it will be interesting to develop a theory of Hodge ideals associated to ideal sheaves (perhaps 
leading to asymptotic constructions as well), to attempt an alternative analytic approach, and to establish connections with constructions 
in positive characteristic generalizing test ideals.

There are other ways in which filtered $\Dmod$-modules underlying Hodge modules have been used in recent years 
in complex and birational geometry, for instance in the study of generic vanishing theorems, holomorphic forms, topological invariants, 
families of varieties and hyperbolicity; see e.g. \cite{DMS}, \cite{Schnell4}, \cite{mhmgv}, \cite{Wang}, \cite{PS}, \cite{PS2}, \cite{PPS}, \cite{Wei}. The bulk of the recent survey \cite{Popa2} treats part of this body of work, so I have decided not to discuss it here again. In any event, the reader is advised to use \cite{Popa2}  as a companion to this article, as introductory material on 
$\Dmod$-modules and Hodge modules together with a guide to technical literature can be found there (especially in Ch. B, C). Much of that will not be repeated here, for space reasons.


\noindent
{\bf Acknowledgements.}
Most of the material in this paper describes joint work with Mircea Musta\c t\u a, and many ideas are due to him. 
Special thanks are due to Christian Schnell, who helped my understanding of Hodge modules through numerous 
discussions and collaborations. I am also indebted to Morihiko Saito, whose work and ideas bear a deep influence on the
topics discussed here. Finally, I thank Yajnaseni Dutta, Victor Lozovanu, Mircea Musta\c t\u a, Lei Wu and Mingyi Zhang for comments.


\subsection{$V$-filtration, Bernstein-Sato polynomial, and birational invariants}\label{general_Vfiltration}
One of the main tools in the theory of mixed Hodge modules is the $V$-filtration along a hypersurface, and its interaction 
with the Hodge filtration. Important references regarding the $V$-filtration include \cite{Kashiwara2}, 
\cite{Malgrange}, \cite{Sabbah}, \cite{Saito-MHP}, \cite{Saito-MV}.

First, let's recall the graph construction.  Let $D$ be an effective divisor on $X$, given (locally, in coordinates
$x_1, \ldots, x_n$) by $f = 0$ with $f \in \shO_X$, and whose support is $Z$. Consider the embedding of $X$ given by the graph of $f$, namely:
$$i_f = ({\rm id}, f) \colon X \hookrightarrow X \times \CC = Y, \,\,\,\,\,\,x \to \big(x, f(x)\big).$$
On $\CC$ we consider the coordinate $t$, and a vector field $\partial_t$ such that $[\partial_t,t]=1$.

Let $(\Mmod, F)$ be a filtered left $\Dmod_X$-module. We denote
$$(\Mmod_f, F) :=  {i_f}_+ (\Mmod, F) = (\Mmod, F) \otimes_{\CC} (\CC[\partial_t], F),$$
where the last equality (which means the filtration is the convolution filtration) is the definition of push-forward for 
filtered $\Dmod$-modules via a closed embedding. More precisely, we have 
\begin{itemize}
\item $\Mmod_f = \Mmod \otimes_{\CC} \CC[\partial_t]$, with action of $\Dmod_Y = \Dmod_X [t, \partial_t]$ given by: 
$\shO_X$ acts by functions on $\Mmod$, and
$$\partial_{x_i}\cdot  (g \otimes \partial_t^i) = \partial_{x_i} g \otimes t^i - (\partial_{x_i}f)g \otimes \partial_t^{i+1},$$
$$t \cdot (g \otimes \partial_t^i) = fg \otimes \partial_t^i - i g\otimes \partial_t^{i-1},\,\,\,\,\,\,{\rm and} \,\,\,\,\,\,
\partial_t \cdot (g \otimes \partial_t^i) = g \otimes \partial_t^{i+1}.$$

\medskip

\item  $F_p \Mmod_f = \bigoplus_{i=0}^p F_{p-i} \Mmod \otimes \partial_t^i$ for all $p \in \ZZ$.
\end{itemize}

One of the key technical tools in the study of $\Dmod$-modules is the $V$-filtration. The $\ZZ$-indexed version always exists and is 
unique when $\Mmod_f$ is a regular holonomic  $\Dmod_Y$-module, due to work of Kashiwara \cite{Kashiwara2} and Malgrange \cite{Malgrange}. Assuming in addition that the local monodromy along $f$ is quasi-unipotent, a condition of Hodge-theoretic origin satisfied by all the objects appearing here, one can also consider the following $\QQ$-indexed version; cf. \cite[3.1.1]{Saito-MHP}.

\begin{definition}[Rational $V$-filtration]
A \emph{rational $V$-filtration}  of $\Mmod_f$ is a decreasing filtration $V^{\alpha} \Mmod_f$ with 
$\alpha \in \QQ$ satisfying the following properties:

\noindent
$\bullet$\,\,\,\,The filtration is exhaustive, i.e. $\bigcup_{\alpha} V^{\alpha} \Mmod_f = \Mmod_f$, and each 
$V^{\alpha} \Mmod_f$ is a coherent $\shO_Y [\partial_{x_i}, \partial_t t]$-submodule of $\Mmod_f$.

\noindent
$\bullet$\,\,\,\,$t\cdot V^{\alpha} \Mmod_f \subseteq V^{\alpha +1} \Mmod_f$ and 
$\partial_t\cdot V^{\alpha} \Mmod_f \subseteq V^{\alpha -1} \Mmod_f$ for all $\alpha \in \QQ$.

\noindent
$\bullet$\,\,\,\,$t\cdot V^{\alpha} \Mmod_f= V^{\alpha+1} \Mmod_f \,\,\,\, {\rm for ~} \alpha > 0$.

\noindent
$\bullet$\,\,\,\,The action of $\partial_t t -\alpha$ on $\gr^{\alpha}_V \Mmod_f$ is nilpotent for each $\alpha$. (One defines 
$\gr_{\alpha}^V \Mmod_f$ as $V^{\alpha} \Mmod_f / V^{> \alpha} \Mmod_f$, 
where $V^{> \alpha} \Mmod_f = \cup_{\beta > \alpha} V^{\beta} \Mmod_f$.)
\end{definition}

We will consider other $\Dmod$-modules later on, but for the moment let's focus on the case $\Mmod = \shO_X$, with the trivial filtration 
$F_k \shO_X = \shO_X$ for $k \ge 0$, and $F_k \shO_X = 0$ for $k < 0$. It is standard to denote $\shB_f : = (\shO_X)_f$. Via 
the natural inclusion of $\shO_X$ in $\shB_f$, for $\alpha \in \QQ$ one defines
$$V^\alpha \shO_X : = V^\alpha \shB_f \cap \shO_X,$$
a decreasing sequence of coherent ideal sheaves on $X$. A first instance of the connections
we focus on here is the following result of Budur-Saito:

\begin{theorem}[{\cite[Theorem~0.1]{BS}}]\label{BS}
If $D$ is an effective divisor on $X$, then for every $\alpha \in \QQ$ one has
$$V^\alpha \shO_X = \I \big((\alpha - \epsilon)D\big),$$
the multiplier ideal  of the $\QQ$-divisor $(\alpha - \epsilon)D$, where $0 < \epsilon \ll 1$ is a rational number.
\end{theorem}

Multiplier ideal sheaves are ubiquitous objects in birational geometry, encoding local numerical invariants of singularities, 
and satisfying Kodaira-type vanishing theorems in the global setting; see \cite[Ch.~9]{Lazarsfeld}.  If $f \colon Y \to X$ 
is a log resolution of the pair $(X, D)$, and $c \in \QQ$, then by definition the multiplier ideal of $cD$ is 
$$\I (c D) = f_* \shO_Y \big(K_{Y/X} - [cf^*D]\big).$$
Let me take the opportunity to also introduce the following notation, to be used repeatedly. Denote $Z = D_{{\rm red}}$, and
define integers $a_i$, $b_i$ and $c_i$ by the expressions
$$f^*Z = \tilde Z + a_1 F_1 + \cdots + a_m F_m$$
and 
$$K_{Y/X} = b_1 F_1 + \cdots + b_m F_m + c_{m+1} F_{m+1} + \cdots + c_n F_n,$$
where $F_j$ are the components of the exceptional locus and $a_i \neq 0$. We denote
\begin{equation}\label{gamma}
\gamma = \underset{1 \le i\le m}{\rm min} \frac {b_i + 1}{a_i}.
\end{equation}

Recall that the \emph{Bernstein-Sato polynomial} of $f$ is the unique monic polynomial $b_f (s)$ of minimal degree, in the variable 
$s$, such that there exists $P\in \Dmod_X [s]$ satisfying the formal identity
$$b_f (s) f^s = P f^{s+1}.$$
See for instance \cite{Kashiwara}, \cite{Sabbah}, \cite{Saito-MV}, while a nice survey can be found in \cite{Granger}.
It can be shown that $b_f (s)$ is independent of the choice of $f$ such that $D = {\rm div}(f)$ locally, and so one also has a function $b_D (s)$ which is 
globally well defined; however, to keep a unique simple notation, in the statements below all information about the pair $(X, D)$ 
related to $b_f(s)$ should be understood locally in this sense.

The roots of the Bernstein-Sato polynomial are interesting invariants of the singularities of $f$, and a number of important 
facts regarding them have been established in the literature. Here are some of the most significant; a posteriori, many of these facts also follow from Theorem \ref{BS} and the connection between the Bernstein-Sato polynomial and the $V$-filtration.
\begin{enumerate}
\item The roots of $b_f (s)$ are negative rational numbers; see \cite{Kashiwara}.
\item More precisely, in the notation above, all the roots of $b_f (s)$ are of the form $- \frac{b_i + \ell}{a_i}$ for some $i \ge 0$ and $\ell \ge 1$; see \cite[Theorem~5]{Lichtin}.
\item The negative $\alpha_f$ of the largest root of $b_f(s)$ is the log canonical threshold of $(X, D)$; \cite[Theorem~10.6]{Kollar}, see also \cite{Yano}, \cite{Lichtin}. 
\item Moreover, all jumping numbers of the pair $(X, D)$ (see \cite[9.3.22]{Lazarsfeld}) in the interval $(0, 1]$ can be 
found among the roots of $b_f (s)$; \cite[Theorem~B]{ELSV}, see also  \emph{loc. cit.} for further references, and \cite{Lichtin2} for a related circle of ideas.
\end{enumerate}

For instance, it is well known that $\alpha_f$ can be characterized in terms of the $V$-filtration as 
$$\alpha_f = {\rm max}~\{\beta \in \QQ~|~ V^\beta \shO_X = \shO_X\},$$
see for instance \cite[(1.2.5)]{Saito-MLCT},
while the log canonical threshold has a similar characterization in terms of the triviality of $\I\big((\beta- \epsilon)D)$.

Assuming that $f$ is not invertible, it is not hard to see that $-1$ is always a root of $b_f (s)$. The polynomial 
$$\tilde{b}_f (s) = \frac{b_f (s)}{s+1}$$
is the \emph{reduced Bernstein-Sato polynomial} of $f$. Inspired by (3) above and a connection with the microlocal $V$-filtration 
\cite{Saito-MV} (cf also \S\ref{Vfiltration}), Saito introduced:

\begin{definition}
The \emph{microlocal log canonical threshold} $\tilde{\alpha}_f$ is the negative of the largest root of the reduced 
Bernstein-Sato polynomial $\tilde{b}_f (s)$. 
\end{definition}

In particular, if $\tilde{\alpha}_f \le 1$, then it coincides with the log canonical threshold. In other words, 
$\tilde{\alpha}_f$ provides a new interesting invariant precisely when the pair $(X, D)$ is log canonical. It is 
already known to be related to standard types of singularities:  

\begin{theorem}\label{Saito_rational}
Assume that $D$ is reduced. Then 
\begin{enumerate}
\item \cite[Theorem~0.4]{Saito-B} ~~$D$ has rational singularities if and only if $\tilde{\alpha}_f > 1$.
\item \cite[Theorem~0.5]{Saito-HF}~~$D$ has Du Bois singularities if and only if $\tilde{\alpha}_f \ge 1$.\footnote{An equivalent 
statement can be found in \cite[Corollary~6.6]{KS}, where it is shown that $D$ is Du Bois if and only if the pair $(X,D)$ is log canonical.}
\end{enumerate}
\end{theorem}


\begin{example}\label{homogeneous}
If $f$ is a weighted homogeneous polynomial such that $x_i$ has weight $w_i$,  the convention being that if $f$ is a sum of monomials 
$x_1^{m_1}\cdots x_n^{m_n}$ then $\sum m_i w_i =1$, we have $\tilde{\alpha}_f  = \sum_{i=1}^n w_i$; see e.g. \cite[4.1.5]{Saito-HF}.
\end{example}

It is well known that the log canonical threshold of the pair $(X, D)$ can be computed in terms of discrepancies; in fact, 
using the notation in ($\ref{gamma}$), given any log resolution one has
$$\alpha_f = {\rm min}\{1, \gamma\}.$$
Similar precise formulas are not known for other roots of the Bernstein-Sato polynomial. Lichtin asked the following regarding the microlocal log canonical threshold.

\begin{question}\cite[Remark 2,~p.303]{Lichtin}\label{Lichtin}
~~Is it true that $\gamma = \tilde{\alpha}_f$?
\end{question}

When $\tilde{\alpha}_f \le 1$, this is indeed the case by the discussion above. As noted by Koll\'ar \cite[Remark~10.8]{Kollar}, the question otherwise has a negative answer, simply due to the fact that in general the 
quantity on the right hand side depends on the choice of log resolution. One of the outcomes of the results surveyed in this paper will be however the inequality $\gamma \le  \tilde{\alpha}_f$; see Theorem \ref{MLCT_bound}. It would be interesting to find similar results for other roots of $\tilde{b}_f (s)$.

\subsection{Hodge filtration on localizations}\label{HF}
I will now start focusing on the Hodge filtration. Saito's theory of mixed Hodge modules produces useful filtered $\Dmod$-modules of geometric and Hodge theoretic origin on complex varieties, which 
extend the notion of a variation of Hodge structure when singularities (of fibers of morphisms, of hypersurfaces, of ambient varieties, etc.) are involved; see for instance the examples 
in \cite[\S2]{Popa2}. Usually the $\Dmod$-module itself is quite complicated, but here we deal with one of the simplest ones. 

Namely, if $X$ is smooth complex variety of dimension $n$, and $D$ is a reduced effective divisor on $X$, we consider the left $\Dmod_X$-module
$$\shO_X (*D) = \bigcup_{k \ge 0} \shO_X (kD)$$
of functions with arbitrary poles along $D$. Locally, if $D = {\rm div}(f)$, then $\shO_X(*D)$ is simply the localization $\shO_X [f^{-1}]$, on which differential operators act by the quotient rule.
This $\Dmod_X$-module underlies the extension of the trivial Hodge module across $D$, i.e. the mixed
Hodge module $j_* \QQ_U^H [n]$, where $U = X \smallsetminus D$ and $j: U \hookrightarrow X$ is the inclusion map. A main feature of $\Dmod$-modules underlying mixed Hodge modules is that they come endowed
with a (Hodge) filtration, in this case $F_k \shO_X (*D)$ with $k \ge 0$, better behaved than those on arbitrary filtered $\Dmod$-modules; besides the fundamental \cite{Saito-MHP}, \cite{Saito-MHM}, 
see also \cite{Schnell1} for an introductory survey, and \cite{SS} for details.

While the $\Dmod$-module $\shO_X (*D)$ is easy to understand, the Hodge filtration can be extremely complicated to describe. This is intimately linked to understanding the singularities of $D$ and the Deligne Hodge filtration on the singular cohomology $H^{\bullet} (U, \CC)$. 
Saito \cite{Saito-B}, \cite{Saito-HF} studied $F_k \shO_X(*D)$ with the help of the $V$-filtration, and established the 
following results:

\begin{theorem}\label{Saito1}
The following hold:
\begin{enumerate}
\item \cite[Proposition~0.9 and Theorem~0.11]{Saito-B} ~~The Hodge filtration is contained in the pole order filtration, namely
$$
F_k \shO_X(*D) \subseteq \shO_X \big( (k+1) D\big) \,\,\,\,\,\, {\rm for ~all} \,\,\,\, k \ge 0,
$$
and equality holds if $k \le \tilde{\alpha}_f - 1$.

\medskip

\item \cite[Theorem~0.4]{Saito-HF} ~~ $F_0 \shO_X(*D)  = \shO_X (D)\otimes_{\shO_X} V^1 \shO_X$.\footnote{This is in fact proved in \emph{loc. cit.} with 
$\widetilde{V}^1 \shO_X$, the microlocal $V$-filtration on $\shO_X$ (see \S\ref{Vfiltration}), instead of $V^1 \shO_X$, but it can be shown that the two coincide for $V^1$.}
\end{enumerate}
\end{theorem}

Item (1) in the theorem leads to defining for each $k \ge 0$ a coherent sheaf of ideals $I_k (D)$ by the formula
$$F_k \shO_X(*D) =  \shO_X \big( (k+1) D\big) \otimes I_k (D).$$
We call these the \emph{Hodge ideals} of the divisor $D$; they, and especially their extensions to $\QQ$-divisors, play the leading role in this note.

\subsection{Review of Hodge ideals for reduced divisors}\label{reduced}
The papers  \cite{MP1} and \cite{MP2} are devoted to the study of Hodge ideals of reduced divisors, using both
properties coming from the theory of mixed Hodge modules, and an alternative approach based on log resolutions 
and methods from birational geometry. 

The theory is essentially complete in this case, and I will only briefly review it in this section (see also \cite[Ch.~F]{Popa2} for a more extensive survey) and in \S\ref{Vfiltration}, where the relationship with the microlocal $V$-filtration \cite{Saito-MLCT} is explained. The rest of the paper discusses the more general case of $\QQ$-divisors, where a complete treatment is only beginning to take shape.

One may loosely summarize the main properties and results as follows:

\begin{theorem}[\cite{MP1}, \cite{MP2}]\label{ideals_reduced}
Given a reduced effective divisor $D$ on a smooth complex variety $X$, the sequence of Hodge ideals $I_k (D)$,
with $k \ge 0$, satisfies:

\noindent
(i) $I_0 (D)$ is the multiplier ideal $\I \big((1 - \epsilon)D\big)$,\footnote{Note that this follows already by combining Theorem \ref{Saito1}(2) and Theorem \ref{BS} above.} so in particular $I_0 (D) = \shO_X$ if and only if the pair $(X, D)$ is log canonical. 
Moreover, there are inclusions
$$\cdots I_k (D) \subseteq \cdots \subseteq I_1(D) \subseteq I_0 (D).$$

\noindent
(ii) When $D$ has simple normal crossings, in a neighborhood where it is given by $x_1\cdots x_r =0$,  $I_k (D)$ is generated by 
$\{x_1^{a_1}\cdots x_r^{a_r} ~|~0 \le a_i \le k, ~ \sum_i a_i = k(r-1) \}$.

\noindent
(iii) $D$ is smooth if and only if $I_k (D) = \shO_X$ for all $k$; cf. also Corollary \ref{smoothness_criterion} below.

\noindent
(iv)  If $I_k (D) = \shO_X$  for some $k\ge 1$ ($\iff I_1 (D) = \shO_X$), then $D$ is normal with rational singularities.  More precisely, $I_1 (D) \subseteq {\rm Adj}(D)$, the adjoint ideal of $D$.\footnote{Recall that $D$ is normal with rational singularities if and only if 
${\rm Adj} (D) = \shO_X$, see \cite[Proposition~9.3.48]{Lazarsfeld}.}

\noindent
(v) There are non-triviality criteria for $I_k (D)$ at a point $x \in D$ in terms of the multiplicity of $D$ at $x$; cf. e.g. Theorem
\ref{nontriviality_reduced} below.

\noindent
(vi) On smooth projective varieties, $I_k(D)$ satisfy a vanishing theorem extending Nadel Vanishing for multiplier ideals 
(a special case of Theorem \ref{vanishing_Hodge_ideals} below).

\noindent
(vii) If $H$ is a smooth divisor in $X$ such that $D|_H$ is reduced, then $I_k (D)$ satisfy 
$$I_k (D|_H) \subseteq I_k(D) \cdot \shO_H,$$
with equality when $H$ is general. 

\noindent
(viii) If $D_1$ and $D_2$ are reduced divisors such that $D_1 +  D_2$ is also reduced, $I_k$ satisfy  the subadditivity property
$$I_k (D_1 + D_2) \subseteq I_k (D_1)\cdot I_k (D_2).$$

\noindent
(ix) If $ X \to T$ is a smooth family with a section $s\colon T \to X$, and $D$ is a relative divisor on $X$ such that 
the restriction $D_t = D|_{X_t}$ to each fiber is reduced, then 
$$\{ t \in T ~|~ I_k (D_t) \subseteq \frak{m}_{s(t)}^q\}$$
is an open subset of $T$, for each $q \ge 1$.

\noindent
(x) $I_k(D)$ determine Deligne's Hodge filtration on the singular cohomology $H^\bullet (U, \CC)$, where $U = X \smallsetminus D$, via a Hodge-to-de Rham type spectral sequence.
\end{theorem}

Note that, in view of item $(i)$, a number of these properties are inspired by well-known properties of multiplier ideals (see \cite[Ch.~9]{Lazarsfeld}), though often the proofs become substantially more involved. However $(ii)$ and $(x)$ simply follow from standard results, via general properties of the Hodge filtration.

Another line of results proved in \cite{MP1} and \cite{MOP} regards the complexity of the Hodge filtration. According to \cite{Saito-HF},  one says that the filtration on a $\Dmod$-module $(\Mmod, F_\bullet)$  is \emph{generated at level $k$} if 
$$F_{\ell} \Dmod_X \cdot F_k \Mmod=F_{k+\ell}\Mmod\quad\text{for all}\quad \ell\geq 0.$$
The smallest integer $k$ with this property is called the \emph{generating level}. In the case of $\Mmod = \shO_X(*D)$ with the Hodge filtration, this can be reinterpreted as saying that 
\begin{equation}\label{ideal_determined}
\shO_X \big((k+\ell + 1)D\big) \otimes I_{k + \ell} (D) = F_{\ell} \Dmod_X \cdot \big(\shO_X \big((k+1)D\big) 
\otimes I_{k} (D)\big),
\end{equation}
so all higher Hodge ideals are determined by $I_k (D)$. Thus this invariant is important for concrete calculations; see e.g. 
Remark \ref{calculations} below.

\begin{theorem}\label{generation_level}
Assume that $X$ has dimension $n$. Then:
\begin{enumerate}
\item \cite[Theorem~B]{MP1} The Hodge filtration on $\shO_X (*D)$ is generated at level $n -2$, and this bound is optimal in general.
\item \cite[Theorem~E]{MOP} If $D$ has only isolated rational singularities and $n \ge 3$, 
then the Hodge filtration on $ \shO_X (*D)$ is generated at level $n -3$.
\end{enumerate}
\end{theorem}

We conjecture in \cite{MOP} that (2) should hold for arbitrary divisors with rational singularities. Its converse is known not to hold in general.
When $D$ has an isolated quasihomogeneous singularity, a stronger bound was given by Saito in \cite[Theorem~0.7]{Saito-HF}: the generating level of $F_\bullet \shO_X(*D)$ is $[n - \tilde{\alpha}_f - 1]$, 
where $\tilde{\alpha}_f$ is the microlocal log
canonical threshold defined in \S\ref{general_Vfiltration}; cf. also Theorem \ref{Saito_rational}(1).

\begin{example}
(1) If $D$ is a reduced divisor on a surface, then the Hodge filtration is generated at level $0$,  so the multiplier ideal $I_0 (D)$ 
determines all other $I_k (D)$ via formula ($\ref{ideal_determined}$) for $k = 0$. See \cite[Example~13.1]{Popa2} for concrete 
calculations.

(2) If $D = (f=0) \subset X = \CC^3$ is a  du Val surface singularity, then $I_0 (D) = \shO_X$, and since $D$ has rational 
singularities, the Hodge filtration is again generated at level $0$. Thus for all $k\ge 1$ we have 
$I_k (D) = f^{k+1} \cdot \big(F_k \Dmod_X  \cdot \frac{1}{f}\big)$.
If however $D$ is an elliptic singularity, then the Hodge filtration is typically not generated at 
level $0$ any more, but only at level $1$. See for instance the elliptic cone calculation in Remark \ref{calculations}.
\end{example}

\noindent 
{\bf Some first applications.}
The use of Hodge ideals in geometric applications is still in its early days. There are however a number of basic consequences 
that can already be deduced using the results above:

\noindent
$\bullet$~~Effective bounds for the degrees of hypersurfaces on which isolated singular points on a reduced hypersurface $D$ in 
$\PP^n$ of fixed degree $d$ impose independent conditions, in the style of a classical a result of Severi for nodal surfaces in $\PP^3$; 
see \cite[\S27]{MP1}. As an example, the isolated singular points on $D$ of multiplicity $m\ge 2$ impose independent conditions on 
hypersurfaces of degree $([\frac{n}{m}]+1)d - n -1$.

\noindent
$\bullet$~~Solution to a conjecture on the multiplicities of points on theta divisors with isolated singularities on principally polarized abelian 
varieties, improving in this case well-known results of Koll\'ar and others; see \cite[\S29]{MP1}. For instance, one shows that every point
on such a theta divisor has multiplicity at most $\frac{g+1}{2}$, where $g$ is the dimension of the abelian variety.

\noindent
$\bullet$~~Effective bound for how far the Hodge filtration coincides with the pole order filtration on the 
cohomology $H^\bullet (U, \CC)$ of the complement $U = X \smallsetminus D$, in the style of results of Deligne, Dimca, Saito and others.
For instance, if $D$ is a divisor with only ordinary singularities of multiplicity $m \ge 2$ in an $n$-dimensional $X$, then  
$$F_p H^\bullet (U, \CC) = P_p H^\bullet (U, \CC) \,\,\,\,\,\,{\rm for ~all} \,\,\,\,\,\, p \le \left[\frac{n}{m}\right] - n - 1.$$
(The two filtrations on $H^\bullet (U, \CC)$ start in degree $-n$.) See \cite[Theorem~D]{MP1}.

Space constraints do not allow me to explain all of this carefully. I will however focus in detail on the second item, and in fact on an extension to pluri-theta divisors in \S\ref{scn:pluritheta}, in order to see the machinery in action.

\subsection{Hodge ideals for arbitrary $\QQ$-divisors}\label{scn:Q-divs}
The case of arbitrary $\QQ$-divisors is treated in \cite{MP3}. It requires a somewhat more technical setting, where the $\Dmod$-modules we consider are only direct summands 
of $\Dmod$-modules underlying mixed Hodge modules. The initial setup can be seen as a $\Dmod$-module analogue of eigensheaf decompositions in the theory of cyclic covering constructions, see e.g. \cite[\S3]{EV}.

Let $D$ be an effective $\QQ$-divisor on $X$, with support $Z$.  We denote $U=X\smallsetminus Z$ and let $j\colon U\hookrightarrow X$ 
be the inclusion map. Locally we can assume that $D=\alpha\cdot {\rm div}(h)$ for some nonzero $h\in \shO_X(X)$ and $\alpha\in\QQ_{>0}$.  We denote $\beta = 1 - \alpha$. 

To this data one associates by a well-known construction the left $\Dmod_X$-module 
$\Mmod(h^{\beta}): = \shO_X (*Z)h^\beta$, a rank $1$ free $\shO_X(*Z)$-module 
with generator the symbol $h^{\beta}$, on which a derivation $D$ of $\shO_X$ acts via the rule
$$D(wh^{\beta}) :=\left(D(w) +  w\frac{\beta\cdot D(h)}{h}\right)h^{\beta}.$$
The case $\beta = 0$  is the localization $\shO_X(*Z)$ considered  in \S\ref{HF}.

This $\Dmod_X$-module does not necessarily itself underlie a Hodge module. It is however a filtered direct summand of one such, 
via the following construction. Let $\ell$ be an integer such that $\ell \beta \in \ZZ$, and consider the finite \'etale map 
$$p \colon V  = {\bf Spec}~ \shO_U [y]/(y^\ell - h^{\ell \beta}) \longrightarrow U.$$
Consider also the cover 
$$q \colon W  = {\bf Spec} ~\shO_X [z]/(z^\ell - h^{\ell \beta}) \longrightarrow X,$$
and a log-resolution $\varphi \colon Y \to W$ of the pair $(W, q^*Z)$ that is an isomorphism over $V$ and is equivariant with respect to 
the natural  $\ZZ/ \ell \ZZ$ action. This all fits in a commutative diagram
$$
\begin{tikzcd}
& Y \dar{\varphi} \arrow[bend left=50]{dd}{g}\\
V \rar\dar{p} & W\dar{q} \\
U\rar{j} & X,
\end{tikzcd}
$$
where the bottom square is Cartesian. Denote by $E$ the support of $g^{-1}(Z)$.

\begin{lemma}\cite{MP3}\label{decomposition}
There is an isomorphism of filtered left $\Dmod_X$-modules
$$g_+ \big(\shO_Y(*E), F_\bullet \big) \simeq j_+ p_+ (\shO_V, F_\bullet) \simeq \bigoplus_{i=0}^{\ell-1}\big(\Mmod(h^{i\beta}), 
F_\bullet \big),$$
where the filtration on the left hand side is given by the pushforward of the Hodge filtration in \S\ref{HF} (cf. also Theorem 
\ref{ideals_reduced}(iii)), while on each summand on the right hand side we consider the induced filtration.
\end{lemma}

For the notation in the lemma, recall that for any proper morphism of smooth varieties $f\colon X \to Y$ there is a 
filtered direct image functor
$$f_+ \colon \DD^b \big({\rm FM}(\Dmod_X)\big) \longrightarrow \DD^b \big({\rm FM}(\Dmod_Y)\big)$$
between the bounded derived categories of filtered $\Dmod$-modules; see \cite[\S2.3]{Saito-MHM}. 

Thus in this theory, the basic (local) object associated to an effective $\QQ$-divisor $D$ as above is the filtered $\Dmod_X$-module
$$\big(\Mmod(h^{\beta}), F_\bullet\big), \,\,\,\,\,\,{\rm with}\,\,\,\,F_k \Mmod (h^{\beta}) \neq 0 \iff k \ge 0,$$
and for most practical purposes this has the same properties as a filtered $\Dmod$-module underlying a mixed Hodge module.

One can show by a direct calculation that when the support $Z$ is \emph{smooth}, itself given by the equation $h$, then 
$$F_k \Mmod(h^{\beta}) = \shO_X \big((k +1 + [\beta]) Z  \big)h^{\beta} \subseteq \shO_X \big((k+1)Z\big)h^{\beta}, \,\,\,\,\,\,{\rm for~all}\,\,\,\,k \ge 0.$$
Using this and standard reduction arguments, it follows that in general (even when $Z$ is not necessarily defined by $h$), 
if $H = {\rm div}(h)$ so that $D = \alpha H$, we have 
$$F_k\Mmod(h^{\beta})\subseteq \shO_X\big(k Z + H\big)h^{\beta}, \,\,\,\,\,\,{\rm for~all}\,\,\,\,k\ge 0.$$

This allows us to formulate the following:

\begin{definition}\label{Hodge_ideals}
For each $k \ge 0$, the \emph{$k$-th Hodge ideal} associated to the $\QQ$-divisor $D$ is defined by 
$$F_k\Mmod(h^{\beta})=  I_k (D) \otimes_{\shO_X} \shO_X\big(kZ + H \big)h^{\beta}.$$
It is standard to check that  the definition of these ideals is independent of the choice of $\alpha$ and $h$, and therefore makes sense globally on $X$.
The reduced case described in \S\ref{HF} and \S\ref{reduced} corresponds to the value $\beta = 0$.
\end{definition}

\noindent
{\bf Assumption.} From now on, for simplicity we assume that $\lceil D \rceil  = Z$ (for instance, $D = \alpha Z$ with $0 < \alpha \le 1$). This makes 
the statements more compact, while the general situation can be reduced to this case by noting that we always have 
$$I_k( D) \simeq I_k (B) \otimes_{\shO_X} \shO_X ( Z - \lceil D \rceil),$$
with $B = Z + D - \lceil D \rceil$.

\medskip

In the rest of this section I will briefly explain the approach to the study of Hodge ideals based on log resolutions, originating in \cite{MP1} in the reduced case, and completed in \cite{MP3} in the general case. In \S\ref{Vfiltration} I will discuss 
the connection with the microlocal 
$V$-filtration discovered by Saito \cite{Saito-MLCT}, and its extension to the twisted case in \cite{MP4}.

Let $f\colon Y\to X$ be a log resolution of the pair $(X, D)$ that is an isomorphism over $U=X\smallsetminus Z$,  
and denote $g=h\circ f\in\shO_Y(Y)$. It is slightly more convenient now to consider equivalently the filtered $\Dmod$-module $\big(\Mmod(h^{-\alpha}), F_\bullet \big)$, with the analogous action of $\Dmod_X$, and 
with a filtered isomorphism given by 
\begin{equation}\label{change}
\Mmod(h^{-\alpha}) \overset{\simeq}{\longrightarrow} \Mmod(h^{\beta}), \,\,\,\,\,\,wh^{-\alpha} \to (wh^{-1})h^{\beta},
\end{equation}
and similarly on $Y$. There is a filtered isomorphism
$$
\big(\Mmod(h^{-\alpha}), F_\bullet \big)\simeq f_+\big(\Mmod(g^{-\alpha}), F_\bullet \big).
$$
We use the notation $G=f^*D$ and $E = {\rm Supp}(G)$, the latter being a simple normal crossing divisor. 
It turns out that there exists a complex on $Y$:
$$C^{\bullet}_{g^{-\alpha}} (-\lceil G\rceil):\,\,0\to\shO_Y(-\lceil G\rceil)\otimes_{\shO_Y} \Dmod_Y\to\shO_Y(-\lceil G\rceil)\otimes_{\shO_Y}\Omega^1_Y(\log E)\otimes_{\shO_Y}\Dmod_Y$$
$$\to\ldots\to\shO_Y(-\lceil G\rceil)\otimes_{\shO_Y}\omega_Y(E)\otimes_{\shO_Y}\Dmod_Y\to 0,$$
which is placed in degrees $-n,\ldots,0$, and such that if $x_1,\ldots,x_n$ are local coordinates, its differential is given by
$$\eta\otimes Q\to d\eta\otimes Q+\sum_{i=1}^n (dx_i\wedge \eta)\otimes\partial_iQ - \alpha \big( {\rm dlog}(g)\wedge\eta\big)\otimes Q.$$
Moreover, this complex has a natural filtration given, for $k\ge 0$, by subcomplexes 
$$F_{k-n}C^{\bullet}_{g^{-\alpha}} (-\lceil G\rceil): = 
0 \rightarrow \shO_Y(-\lceil G\rceil)\otimes  F_{k-n} \Dmod_Y \rightarrow $$
$$\to  \shO_Y(-\lceil G\rceil)\otimes\Omega_Y^1(\log E) \otimes  F_{k-n+1} \Dmod_Y \rightarrow 
\cdots \to  \shO_Y(-\lceil G\rceil)\otimes \omega_Y(E) \otimes  F_k \Dmod_Y\rightarrow 0.$$

\noindent
The key point shown in \emph{loc. cit.} is that there is a filtered quasi-isomorphism
\begin{equation}\label{quasi_SNC}
\big(C^{\bullet}_{g^{-\alpha}} (-\lceil G\rceil), F_\bullet\big) \simeq \big(\Mmod_r (g^{-\alpha}), F_\bullet\big),
\end{equation}
where 
$$\Mmod_r (g^{-\alpha}) : = \Mmod (g^{-\alpha}) \otimes_{\shO_Y} \omega_Y \simeq g^{-\alpha} \omega_Y(*E)$$ 
is the filtered right $\Dmod_Y$-module associated to $\Mmod (g^{-\alpha})$.
In other words, the filtered complex on the left computes the Hodge filtration on $\Mmod_r (g^{-\alpha})$, 
hence the Hodge ideals for the simple normal crossings divisor $E$.

Given this fact, one can use $\big(C^{\bullet}_{g^{-\alpha}} (-\lceil G\rceil), F_\bullet\big)$ as a concrete representative for computing the filtered $\Dmod$-module pushforward of $\big(\Mmod_r (g^{-\alpha}), F_\bullet\big)$, hence for computing 
the ideals $I_k (D)$. If we denote as customary by 
$$\Dmod_{Y\to X} = \shO_Y \otimes_{f^{-1}\shO_X} f^{-1} \Dmod_X$$
the transfer $\Dmod$-module (isomorphic to $f^*\Dmod_X$ as an $\shO_Y$-module), the result is:

\begin{theorem}\cite{MP3}\label{formula_log_resolution}
With the above notation, the following hold:
\begin{enumerate}
\item For every $p\neq 0$ and every $k\in \ZZ$, we have
$$R^pf_*\big(C^{\bullet}_{g^{-\alpha}} (-\lceil G\rceil)\otimes_{\Dmod_Y}\Dmod_{Y\to X}\big)=0\,\,\,\,
{\rm and} \,\,\,\,  R^pf_*F_k\big(C^{\bullet}_{g^{-\alpha}} (-\lceil G\rceil)
\otimes_{\Dmod_Y}\Dmod_{Y\to X}\big)=0.$$
\item  For every $k\in\ZZ$, the natural inclusion induces an injective map
$$ R^0f_*F_k\big(C^{\bullet}_{g^{-\alpha}} (-\lceil G\rceil)\otimes_{\Dmod_Y}\Dmod_{Y\to X}\big)\hookrightarrow R^0f_*\big(C^{\bullet}_{g^{-\alpha}} (-\lceil G\rceil)\otimes_{\Dmod_Y}\Dmod_{Y\to X}\big).$$
\item We have a canonical isomorphism
$$R^0f_*\big(C^{\bullet}_{g^{-\alpha}} (-\lceil G\rceil)\otimes_{\Dmod_Y}\Dmod_{Y\to X}\big)\simeq \Mmod_r(h^{-\alpha})$$ 
that, using $(2)$ and ($\ref{change}$), induces for every $k\in\ZZ$ an isomorphism
$$R^0f_*F_{k-n}\big(C^{\bullet}_{g^{-\alpha}} (-\lceil G\rceil)\otimes_{\Dmod_Y}\Dmod_{Y\to X}\big)\simeq
h^{\beta}\omega_X\big(kZ + H \big)\otimes_{\shO_X}I_k(D) = F_{k-n} \Mmod_r (h^{\beta}).$$
\end{enumerate}
\end{theorem}

\smallskip

\begin{example}[$I_0 (D)$ is a multiplier ideal]\label{I_0}
The lowest term in the filtration on the complex above reduces to the sheaf
$$F_{-n} C^{\bullet}_{g^{-\alpha}} (-\lceil G\rceil)= \omega_Y(E-\lceil f^*D\rceil)$$
in degree $0$. Thus 
$$I_0 (D) = f_*\shO_Y\big(K_{Y/X}+E-\lceil f^*D\rceil\big)=  f_*\shO_Y\big(K_{Y/X}- [ (1 - \epsilon)f^*D]\big).$$
This is by definition the multiplier ideal associated to the $\QQ$-divisor $(1-\epsilon) D$ with $0<\epsilon\ll 1$. Consequently (see  \cite[9.3.9]{Lazarsfeld}):
$$I_0 (D) = \shO_X \iff (X, D) {\rm~~is~log~canonical}.$$
\end{example}

\begin{remark}[Local vanishing]
In view of Theorem \ref{formula_log_resolution}(3) and Example \ref{I_0}, the statement in 
Theorem \ref{formula_log_resolution}(1) can be seen as a generalization of Local Vanishing for multiplier ideals 
\cite[Theorem~9.4.1]{Lazarsfeld}.
\end{remark}

Given the equivalence between the triviality of $I_0 (D)$ and log canonicity, it is natural to introduce the following:

\begin{definition}\label{k-log-can}
We say that the pair $(X, D)$ is $k$-log canonical if 
$$I_0 (D) = \cdots = I_k (D) = \shO_X.$$
Under our running assumption on $D$, Corollary \ref{inclusions} below implies that this is in fact equivalent to 
simply asking that $I_k (D) = \shO_X$.
\end{definition}

\begin{example}
Let $Z$ have an ordinary singularity of multiplicity $m$, i.e. an isolated singular point whose projectivized tangent cone is smooth (for example the cone over a smooth hypersurface of degree $m$ in $\PP^{n-1}$). If $D = \alpha Z$ with $0 < \alpha \le 1$, then $(X, D)$ is $k$-log canonical if and only if $k \le [\frac{n}{m} - \alpha]$. 
See Corollary \ref{mlc}, noting that $\tilde\alpha_f = \frac{n}{m}$ according to \cite{Saito-HF}; cf. also \cite[Theorem~D and Example 20.13]{MP1}.
\end{example}

\begin{example}
Irreducible theta divisors on principally polarized abelian varieties are $0$-log canonical, but may sometimes not be $1$-log canonical; see \cite[Remark~29.3(2)]{MP1}. Generic determinantal hypersurfaces are $1$-log canonical, but they are not 
$2$-log canonical; see \cite[Example~20.14]{MP1}. Both have rational singularities; compare with Theorem \ref{ideals_reduced}(iv).
\end{example}

The generation level of the Hodge filtration on $\Mmod(h^\beta)$ is not so well understood at the moment; for instance, depending on the value of $\alpha$, examples in \cite{MP3} show that on surfaces it can be either $0$ or $1$. It is natural to ask what is the precise analogue of Theorem \ref{generation_level} (note however that we show in \emph{loc.cit.} that the generation level is always at most $n-1$), but also, concretely, whether the analogue of Saito's result discussed 
immediately after it holds:

\begin{question}
If $D = \alpha Z$, with $Z$ reduced and having an isolated quasi-homogeneous singularity, is the generation level of the Hodge filtration on $\Mmod(h^\beta)$ equal to $[n - \tilde\alpha_f - \alpha]$?\footnote{Note added after the final proofs: in the meanwhile, this was verified by M. Zhang \cite{Zhang:Vfiltration}.}
\end{question}

\subsection{(Non)triviality criteria}
The applications of the theory of multiplier ideals rely crucially on effective criteria for understanding whether they are trivial 
or not at a given point. The most basic are as follows; if $D$ is an effective $\QQ$-divisor, then:
\begin{enumerate}
\item If ${\rm mult}_x (D) \ge n = \dim X$, then $\I(D)_x \neq \shO_{X,x}$; see \cite[Proposition~9.3.2]{Lazarsfeld}.
\item If ${\rm mult}_x (D) < 1$, then $\I(D)_x = \shO_{X,x}$; see \cite[Proposition~9.5.13]{Lazarsfeld}.
\end{enumerate}
The first is quite standard, while the second is a slightly more delicate application of inversion of adjunction. 

Multiplier ideals also satisfy a birational transformation formula. If $f \colon Y \to X$ is any proper birational map, 
then 
$$\I (D) \simeq f_* \big(\shO_Y (K_{Y/X}) \otimes_{\shO_Y} \I (f^*D)\big).$$
See \cite[Theorem~9.2.33]{Lazarsfeld}. Such a compact statement is not available for higher Hodge ideals; however, using Theorem \ref{formula_log_resolution}, one can show a partial analogue. 

\begin{theorem}\cite[Theorem~18.1]{MP1}, \cite{MP3}\label{smaller_ideal}
Let $f \colon Y \to X$ be a projective morphism, with $Y$ smooth. Let $Z = D_{\rm red}$, $E = (f^*D)_{\rm red}$,  and denote
$T_{Y/X} =  {\rm Coker}(T_Y \to f^*T_X)$. Then:
\begin{enumerate}
\item There is an inclusion
$$f_*\big(I_k(f^*D)\otimes_{\shO_{Y}}\shO_{Y}(K_{Y/X}+ k(E- f^*Z))\big)\subseteq I_k(D).$$
\item  If $J$ is a coherent ideal on $X$ such that $J\cdot T_{Y /X}=0$, then 
$$J^k\cdot I_k(D)\subseteq f_*\big(I_k(f^*D)\otimes_{\shO_{Y}}\shO_{Y}(K_{Y/X}+ k (E-f^*Z))\big).$$
\end{enumerate}
\end{theorem}

The first statement leads quite quickly to the following triviality criterion, in terms of the coefficients of exceptional divisors on a fixed log resolution.

\begin{corollary}\label{triviality_criterion}
Assume that $D = \alpha Z$ (with $0 < \alpha \le 1$) and for $f\colon Y \to X$ a log resolution of the pair $(X, D)$, define 
$\gamma$ as in  ($\ref{gamma}$). If 
$$ \gamma \ge k + \alpha,$$
then  $I_k (D) = \shO_X$.
\end{corollary}

This is a key ingredient in bounding the microlocal log canonical threshold of $D$ in terms discrepancies; see Theorem \ref{MLCT_bound} below.
 
On the other hand, the second statement in Theorem \ref{smaller_ideal} leads to nontriviality criteria that, just as in the case of 
multiplier ideals, are useful when combined with global statements like the vanishing theorem explained in the next section.

\begin{corollary}\label{inclusion_power}
If $x \in X$ is such that ${\rm mult}_x Z=a$ and 
${\rm mult}_x D=b$, and if $q$ is a non-negative integer such that
$$b+ka>q+r+2k-1,$$
then $I_k(D)\subseteq \frak{m}_x^q$. In particular, this happens if ${\rm mult}_WD>\frac{q+r+2k-1}{k+1}$. 
\end{corollary}

At least for the moment, one can obtain somewhat stronger statements in the reduced case; the following collects some of the results in \cite{MP1}. The proofs are more involved, (1) relying for instance on a deformation to ordinary singularities argument 
using Theorem \ref{ideals_reduced}(ix), combined with explicit calculations in that case.

\begin{theorem}\label{nontriviality_reduced}
If $x \in D$ is a point on a reduced divisor, with $m = {\rm mult}_x (D)$, then:
\begin{enumerate}
\item $I_k (D) \subseteq \frak{m}_x^q$ if $q = {\rm min}\{m-1,(k+1)m - n\}$; see \cite[Theorem~E]{MP1}.
\item $I_k (D) \subseteq \frak{m}_x^q$ if $m \ge 2 +  \frac{q + n - 2} {k+1}$; see \cite[Corollary~19.4]{MP1}.
\end{enumerate}
\end{theorem}

As an example, for $k =1$ the criterion in $(1)$ can be rephrased as
$$m \ge {\rm max} \left\{q + 1, \frac{n+q}{2}\right\} \implies I_1 (D) \subseteq \frak{m}_x^q.$$
It also implies that if $x \in D$ is a \emph{singular} point, then
$$I_k (D) \subseteq \frak{m}_x,\,\,\,\,\,\,{\rm for~all}\,\,\,\,k \ge \frac{n-1}{2}.$$

In particular one obtains a smoothness criterion in terms of the Hodge filtration:

\begin{corollary}\cite[Theorem~A]{MP1}\label{smoothness_criterion}
The divisor $D$ is smooth $\iff$ $I_k (D) = \shO_X$ for all $k$ $\iff$ $I_k (D) = \shO_X$ for some $k \ge  \frac{n-1}{2}$.
\end{corollary}

\subsection{Global setting and vanishing theorem}
While the locally defined ideals in Definition \ref{Hodge_ideals} glue together into a global object, this is not usually the case with 
the $\Dmod$-modules $\Mmod(h^{\beta})$. There is however a setting in which this can be done. 

Namely, assume that  $D=\frac{1}{\ell}H$, where $H$ is an integral divisor and $\ell$ is a positive integer, and that 
there is a line bundle $M$ such that $\shO_X(H)\simeq M^{\otimes \ell}$. (This of course always holds when $D$ is integral.)
Let $s\in \Gamma(X,M^{\otimes \ell})$ be a section whose zero-locus is $H$. 
Recall that $U = X \smallsetminus Z$, and $j \colon U \hookrightarrow X$ is the inclusion.
Since $s$ does not vanish on $U$, we may consider the section $s^{-1}\in\Gamma(U,(M^{-1})^{\otimes \ell})$. Let 
$$p\colon V = {\bf Spec}\big(\shO_X\oplus M\oplus\ldots\oplus M^{\otimes(\ell-1)}\big)\longrightarrow U$$
be the \'{e}tale cyclic cover corresponding to $s^{-1}$.
The filtered $\Dmod_X$-module 
$$(\Mmod, F_\bullet)=j_+p_+ (\shO_V, F_\bullet)$$ 
underlies a mixed Hodge module, and the obvious $\mu_{\ell}$-action on $\Mmod$ induces an eigenspace decomposition
$$(\Mmod, F_\bullet)=\bigoplus_{i=0}^{\ell-1}(\Mmod_i, F_\bullet),$$
where $\Mmod_i$ is the eigenspace corresponding to the map $\lambda\to\lambda^i$, and 
on each $\Mmod_i$ we consider the induced filtration.

On open subsets $W$ on which $M$ is trivialized we have isomorphisms of filtered $\Dmod_W$-modules
$$\Mmod_i\simeq\Mmod(s_{|W}^{-i/\ell})\quad\text{for}\quad 0\leq i\leq \ell-1,$$
which glue to a global isomorphism 
$$\Mmod_i\simeq M^{\otimes i}\otimes_{\shO_X}\shO_X(*Z)=j_*j^*M^{\otimes i}.$$
Twisting in order to globalize the  $\Mmod(h^{\beta})$ picture, with $\beta= 1 - \frac{1}{\ell}$, 
we obtain global coherent ideals given by
$$F_k\Mmod_i \simeq M^{\otimes i} (-H) \otimes_{\shO_X}I_k\left(i/\ell \cdot H\right)\otimes_{\shO_X}\shO_X\big(kZ + H\big),$$
and the Hodge ideals $I_k (D)$ are defined by the case $i = 1$.

In this global setting, there is a vanishing theorem for Hodge ideals that in the case $k =0$ is nothing else but the celebrated 
Nadel vanishing theorem for multiplier ideals. This was shown in \cite[Theorem~F]{MP1} in the reduced case, and in \cite{MP3} in general. Recall that here we are assuming $\lceil D \rceil = Z$, the support of $D$, for simplicity.

\begin{theorem}\label{vanishing_Hodge_ideals}
Assume that $X$ is a smooth projective variety of dimension $n$,  and $D$ is a $\QQ$-divisor  as at the beginning of this section. 
Let $L$ a line bundle on  $X$ such that $L + Z - D$ is ample. For some $k \ge 0$, assume that the pair $(X,D)$ 
is $(k-1)$-log-canonical, i.e. $I_0 (D) = \cdots = I_{k-1} (D) = \shO_X$.
Then we have:

\begin{enumerate}
\item If $k \le n$, and $L (pZ)$ is ample for all $1 \le p \le k$, then 
$$H^i \big(X, \omega_X \otimes L ((k+1)Z)  \otimes I_k (D) \big) = 0$$
for all $i \ge 2$. Moreover, 
$$H^1 \big(X, \omega_X \otimes L ((k+1)Z) \otimes I_k (D) \big) =  0$$
holds if $H^j \big(X, \Omega_X^{n-j} \otimes L ((k - j +1)Z)\big) = 0$ for all $1 \le j \le k$.
\medskip

\item If $k \ge n+1$ and $L ((k+1) Z)$ is ample, then 
$$H^i \big(X, \omega_X \otimes L ((k+1)Z) \otimes I_k (D) \big) = 0 \,\,\,\,\,\, {\rm for ~all}\,\,\,\, i >0.$$

\medskip

\item If $D + pZ$ is ample for $0 \le p \le k-1$, then (1) and (2) also hold with $L = M (- Z)$.
\end{enumerate}
\end{theorem}

The main ingredient in the proof is Saito's Kodaira-type vanishing theorem \cite[\S2.g]{Saito-MHM} for mixed Hodge modules, 
stating that if 
$(\Mmod, F_\bullet)$ is the filtered $\Dmod$-module underlying a mixed Hodge module on a projective variety $X$, then
$${\bf H}^i \big(X, \gr_k^F {\rm DR}(\Mmod, F_\bullet) \otimes L \big) = 0 \,\,\,\,\,\,{\rm for~all}\,\,\,\,i >0,$$
where $L$ is any ample line bundle, and $\gr_k^F {\rm DR}(\Mmod, F_\bullet)$ denotes for each $k$ the associated graded of the 
induced filtration on the de Rham complex of $\Mmod$. See \cite{Schnell2}, \cite{Popa} for more on this theorem, and also \cite[\S3]{Popa2} for a guide to interesting generalizations. In (3), this is replaced by Artin vanishing (on affine varieties) for the 
perverse sheaf associated to $\Mmod$ via the Riemann-Hilbert correspondence.

\begin{remark}\label{better_vanishing}
When $X$ has cotangent bundle with special properties, for instance when it is an abelian variety or $\PP^n$ (or more generally a 
homogeneous space), the hypotheses on $(k-1)$-log canonicity and borderline Nakano-type vanishing are not needed, so 
vanishing holds in a completely arbitrary setting; see for instance \cite[\S25,~\S28]{MP1}. Similarly, stronger vanishing holds on 
toric varieties \cite{Dutta}.
\end{remark}

It will be important to address the following natural problem for non-reduced divisors:

\begin{question}
Does vanishing for $\QQ$-divisors hold without the global assumption on the existence of $\ell$-th roots of $\shO_X(H)$ at the beginning of the section?
\end{question}

\subsection{Example of application: pluri-theta divisors on abelian varieties}\label{scn:pluritheta}
The goal here is to see explicitly how the combination of local nontriviality criteria and global vanishing for Hodge ideals 
can be put to use towards concrete applications. I will focus on one example: divisors in pluri-theta linear series on principally polarized abelian varieties. The statement below partially extends the result \cite[Theorem~I]{MP1} about theta divisors; the general idea is quite similar.

Let $(A, \Theta)$ be a principally polarized abelian variety of dimension $g$. Let $D\in |n \Theta|$ for some $n \ge 1$, whose support $Z$ has only isolated singularities.\footnote{If $D\in |n \Theta|$ is \emph{any} divisor, then \cite[Proposition 3.5]{EL} implies that ${\rm mult}_x D \le ng$ for arbitrary $D$, 
which is stronger than one of the results I had originally listed here.}

\begin{theorem}\label{pluritheta}
Under the hypotheses above, if $\epsilon (\Theta)$ denotes the Seshadri constant of $\Theta$, and $x \in D$, we have:

\begin{enumerate}
\item If $A$ is general in the sense that $\rho(A) = 1$, then ${\rm mult}_x D \le  n^2 \epsilon (\Theta)+ n$.
\item If $D$ is reduced, then ${\rm mult}_x D \le  n \epsilon (\Theta)+ 1$.
\end{enumerate}
\end{theorem}

Recall that the definition of the Seshadri constant easily implies that $\epsilon (\Theta) \le \sqrt[g]{g!}$, see \cite[Proposition 5.1.9]{Lazarsfeld}, though often its value can be much lower.  Before proving the theorem, let's introduce the notation $s (\ell, x)$ for  the largest 
integer $s$ such that the linear system $|\ell \Theta|$ separates $s$-jets at $x$, i.e. such that the restriction map
$$H^0 \big(A, \shO_A (\ell \Theta) \big) \longrightarrow H^0 \big(A, \shO_A (\ell \Theta) \otimes \shO_A /\mathfrak{m}_x^{s+1} \big)$$ 
is surjective. A basic fact is that 
\begin{equation}\label{seshadri_inequality}
\frac{s (\ell, x)}{\ell} \le \epsilon (\Theta, x),
\end{equation}
the Seshadri constant of $\Theta$ at $x$, and that $\epsilon (\Theta, x)$ is the limit of these quotients as $\ell \to \infty$; see \cite[Theorem 5.1.17]{Lazarsfeld} and its proof. Since $A$ is homogeneous, $\epsilon (\Theta, x)$ does not actually depend on $x$, so it is denoted  
$\epsilon (\Theta)$.

\begin{proof}[Proof of Theorem \ref{pluritheta}]
We prove (1), and at the end indicate the necessary modification needed to deduce (2).

Write $D = \sum a_i Z_i$, with $Z_i$ prime divisors,  so that $Z = \sum Z_i$. 
For each $i$ we have 
\begin{equation}\label{inequality2}
a_i Z_i \cdot \Theta^{g-1} \le H \cdot \Theta^{g-1} = n \cdot g!.
\end{equation}
On the other hand, the assumption on $A$ implies that $Z_i \cdot \Theta^{g-1} \ge g!$,
hence it follows that $a_i \le n$. Thus $D \le n Z$, so if $m = {\rm mult}_x (D)$, then 
$${\rm mult}_x (Z) \ge \lceil \frac{m}{n} \rceil.$$
Since $x$ is fixed, for simplicity we denote $s_\ell = s (\ell, x)$. I claim that 
\begin{equation}\label{inequality}
\frac{m}{n} \le \lceil \frac{m}{n}  \rceil \le \frac{(s_{n(k+1)} + g + k +1)}{k+1}, \,\,\,\,\,\,{\rm for~all}\,\,\,\,k \ge 1.
\end{equation}
Assuming the opposite inequality for some $k$, by Theorem \ref{nontriviality_reduced}(2) we have 
$$I_k (Z) \subseteq \frak{m}_x^{s_{n (k+1)} + 2}.$$
Now according to the vanishing in \cite[Theorem~28.2]{MP1}, a refinement on abelian varieties of the statement of 
Theorem \ref{vanishing_Hodge_ideals} (cf. Remark \ref{better_vanishing}), we have: 
$$H^1 \big(A, \shO_A((k+1)Z) \otimes \alpha \otimes I_k (Z) \big) = 0$$
for every $\alpha \in \Pic^0 (A)$. It is clear that we can write 
$$(k+1)D = (k+1)Z + N,$$ 
where $N$ is a nef divisor on $A$ (in fact either $0$ or ample). On the other hand, nef line bundles on abelian 
varieties are special examples of what are called $GV$-sheaves (a condition involving the Fourier-Mukai transform, see e.g. 
\cite[\S2]{PP5}), so we conclude using \cite[Proposition~3.1]{PP5}\footnote{The local freeness condition in the statement in \emph{loc. cit.} is not needed in its proof.} that we have 
$$H^1 \big(A, \shO_A ( n(k+1) \Theta) \otimes \alpha \otimes I_k (Z) \big) = 0 \,\,\,\,\,\,{\rm for~all}\,\,\,\, \alpha \in \Pic^0 (A).$$

Going back to the inclusion $I_k (Z) \subseteq \frak{m}_x^{s_{n(k+1)} +2}$, since $Z$ has only isolated 
singularities, the quotient $\frak{m}_x^{s_{n(k+1)} +2}/ I_k (Z)$ is supported in dimension $0$. We obtain 
$$H^1 \big(A, \shO_A (n(k+1)\Theta) \otimes \alpha \otimes \frak{m}_x^{s_{n (k+1) +2}} \big) = 0$$
for every $\alpha \in \Pic^0 (A)$. 
But the collection of line bundles $\shO_A (n(k+1)\Theta) \otimes \alpha$ is, as $\alpha$ varies in $\Pic^0(X)$, the same as the collection of line bundles $t^*_a \shO_A (n(k+1)\Theta)$ as $a$ varies in $X$, where $t_a$ denotes translation by $a$. Therefore the
vanishing above is equivalent to the statement that $|n(k+1)\Theta|$ separates $(s_{n(k+1)} + 1)$-jets, which gives a contradiction and  proves (\ref{inequality}).

Finally, since $s_{n(k+1)} \le  n(k+1) \epsilon (\Theta)$ by ($\ref{seshadri_inequality}$), we deduce that
$$m \le  \frac{n( n(k+1) \epsilon (\Theta) + g+ k +1)}{k+1}, \,\,\,\,\,\,{\rm for~all}\,\,\,\,k \ge 1.$$
Letting $k \to \infty$,  we obtain the inequality in the statement.

For the statement in (2), note that under the extra assumption we have $D = Z$, so we know directly that ${\rm mult}_x (Z) = m$.
Therefore we don't need to divide by $n$ in all the formulas above, while the rest of the argument is completely identical.
\end{proof}

\subsection{$V$-filtration and microlocal log-canonical threshold}\label{Vfiltration}
In this final section I turn to the connection between Hodge ideals and the $V$-filtration, first noted in \cite{Saito-MLCT}. For a $\QQ$-divisor $D$ on $X$, defined locally as $D = \alpha \cdot {\rm div}(f)$, just as in \S\ref{scn:Q-divs} we assume for simplicity that $Z = {\rm div}(f)$ is the reduced structure on $D$, and that $0 < \alpha \le 1$. The corresponding statements for arbitrary $D$ can be found in \cite{MP4}.

We return to the notation introduced in \S\ref{general_Vfiltration}. Recall that for a $\Dmod_X$-module $\Mmod$, we denote by 
$\Mmod_f$ its pushforward via the graph of $f$.  In line with \cite{Saito-B} and 
\cite{Saito-MLCT}, I will use the notation
$$\shB_f : = (\shO_X)_f, \,\,\,\,\,\,\shB_f(*Z) : = (\shO_X (*Z))_f,  \,\,\,\,\,\,{\rm and}\,\,\,\,\,\, \shB_f^\beta(*Z) : = (\shO_X (*Z)f^{\beta})_f.$$

One can use the $V$-filtration on $\shB_f$ in order to define some interesting ideals on $X$ associated to $D$.

\begin{definition}
For each $k\ge 0$, we define 
$$\tilde{I}_k (D) : = \{v \in \shO_X ~|~ \exists~ v_0, v_1, \ldots , v_k = v \in \shO_X {\rm ~such~that~} \sum_{i = 0}^k v_i \otimes \partial_t^i \in V^\alpha \shB_f\}\subseteq \shO_X.$$
Since $0 < \alpha \le 1$, this is just another way of writing the filtration 
$\widetilde{V}^{\bullet} \shO_X$ induced 
on $\shO_X$ by Saito's \emph{microlocal $V$-filtration} \cite{Saito-MV}, \cite{Saito-MLCT}. In the notation of \emph{loc. cit.}, we have
\begin{equation}\label{micro}
\tilde{I}_k (D) = \widetilde{V}^{k + \alpha} \shO_X.
\end{equation}
\end{definition}

When $D = Z$ is a reduced divisor (i.e. $\alpha =1$), a comparison theorem between Hodge ideals and these ``microlocal" ideals was established  recently by Saito.

\begin{theorem}\cite[Theorem~1]{Saito-MLCT}\label{saito}
If $D$ is reduced, then for every $k \ge 0$ we have
$$I_k (D) = \tilde{I}_k (D)  \,\,\,\,{\rm mod}~f.$$
\end{theorem}

The statement means that the equality happens only in the quotient $\shO_D$. For $k =0$ it holds without modding out by $f$, by Theorem \ref{BS}. However, for higher $k$ it does not necessarily hold in $\shO_X$; see Remark \ref{calculations}.

Its extension to arbitrary $\QQ$-divisors is established in \cite{MP4}, 
as a consequence of a statement which is more explicit, in the sense of completely computing Hodge ideals in terms of the $V$-filtration, even in the reduced case. For $i \ge 0$, we denote 
$$Q_i(X)=\prod_{j=0}^{i-1}(X + j) \in \ZZ[X].$$

\begin{theorem}\cite{MP4}\label{general_description}
If $D$ is a $\QQ$-divisor as above, then  for every $k \ge 0$ we have
$$I_k (D)=\left\{\sum_{j=0}^p Q_j(\alpha)f^{p-j}v_j~ | ~\sum_{j=0}^pv_j\otimes\partial_t^j\delta\in 
V^{\alpha} \shB_f \right\}.$$
In particular, we have
$$I_k (D)  = \tilde{I}_k (D)  \,\,\,\,{\rm mod}~f.$$
\end{theorem}

One of the key technical points in \cite{MP4} is a description of the $V$-filtration on $\shB_f^\beta (*Z)$ in terms of that 
on $\shB_f(*Z)$, based on Sabbah's computation of the $V$-filtration in terms of the Bernstein-Sato polynomials of 
individual elements in the $\Dmod$-module \cite{Sabbah}.

Theorem \ref{general_description} has consequences regarding the basic behavior of Hodge ideals that, surprisingly, at the moment are not 
known by other means. Recall for instance the chain of inclusions in Theorem \ref{ideals_reduced}(i); this seems unlikely to hold in the general $\QQ$-divisor 
case, but the following is nevertheless true, given ($\ref{micro}$).

\begin{corollary}\label{inclusions}
For each $k \ge 1$ we have 
$$I_k (D) + (f) \subseteq I_{k-1} (D) + (f).$$
\end{corollary}

Stronger statements hold for the first nontrivial ideal, as it is not hard to see that the $k$-log-canonicity of a divisor $D$ (see Definition \ref{k-log-can}) implies that
$(f) \subseteq I_{k+1}(D)$.

\begin{corollary}
If $(X, D)$ is $(p-1)$-log canonical, then 
$$\tilde{I}_p (D) \subseteq I_p (D) = \tilde{I}_p (D) + (f)$$ 
and also
$$I_{p+1} (D) \subseteq I_p (D).$$
In particular, we always have $I_1 (D) \subseteq I_0 (D)$.
\end{corollary}

Another important consequence regards the behavior of the Hodge ideals $I_k (\alpha Z)$ when $\alpha$ varies. In the case of $I_0$, it is well known that they get smaller as $\alpha$ increases, and that there is a discrete set of values of $\alpha$ (called jumping coefficients) where there the ideal actually changes; see \cite[Lemma ~9.3.21]{Lazarsfeld}. This is not the case for higher $k$; for the cusp 
$Z = (x^2 + y^3 = 0)$ and $5/6 < \alpha \le 1$, one can see that 
$$I_2 (\alpha Z) = (x^3, x^2y^2, xy^3, y^4- (2 \alpha + 1) x^2 y),$$
and thus we obtain incomparable ideals. However, Theorem \ref{general_description} implies that the picture does becomes similar to that for multiplier ideals if one considers the images in $\shO_D$.

\begin{corollary}
Given any $k$, there exists a finite set of rational numbers $0 = c_0 < c_1 < \cdots < c_s < c_{s+1} = 1$ such that 
for each $0 \le i \le s$ and each $\alpha \in (c_i, c_{i+1}]$ we have 
$$I_k (\alpha Z)~{\rm mod} ~f = I_k (c_{i+1} Z)~{\rm mod} ~f = {\rm constant}$$
and such that 
$$I_k (c_{i+1} Z) ~{\rm mod} ~f\subseteq I_k (c_i Z) ~{\rm mod} ~f.$$
\end{corollary}

In fact, for a fixed $k$, the set of $c_i$ is contained in the set of jumping coefficients for the $V$-filtration on $\shB_f$
 in the interval $(k, k+1]$.

\begin{remark}[Calculations]\label{calculations}
There are also significant computational consequences; indeed, in \cite[\S2.2-2.4]{Saito-MLCT}, Saito fully computes the microlocal $V$-filtration for weighted-homogeneous isolated singularities. For example,  in the case of diagonal hypersurfaces $f = x_1^{a_1} + \cdots + x_n^{a_n}$ (which was previously obtained in \cite[Example~2.6]{MSS} using a Thom-Sebastiani type theorem), 
$\widetilde{V}^\alpha$ is generated by monomials of the form $x_1^{\nu_1}\cdots x_n^{\nu_n}$ satisfying
$$\sum_{i=1}^n \frac{1}{a_i} \left(\nu_i + 1 + \left[\frac{\nu_i}{a_i - 1}\right]\right)\ge \alpha.$$
Saito also shows in \emph{loc. cit.} that $I_1(D) = \tilde{I}_1 (D)$ in the reduced homogeneous case, though this typically fails for $k \ge 2$. Consider as an example the elliptic cone $D = (x^3 + y^3 + z^3 = 0) \subseteq \AAA^3$. The pair is log canonical, hence $I_0 (D) = \shO_X$. Moreover, it follows from the above that 
$$I_1(D) = \tilde{I}_1 (D) = (x^2, y^2, z^2, xyz).$$
Theorem \ref{generation_level}(1) implies that from this one can compute all other $I_k (D)$.  The calculations 
in \cite{Saito-MLCT} show however that the element $- 2x^4 + xy^3 + xz^3$ belongs to $I_2 (D)$, but not to $\tilde{I}_2 (D)$.
Many concrete calculations of Hodge ideals can also be performed based on the results in \cite{Saito-HF}; see also the upcoming 
\cite{Zhang:Vfiltration} for generalizations to $\QQ$-divisors.
\end{remark}

\noindent
{\bf Microlocal log canonical threshold.}
Part of the usefulness of the results  above stems from the connection between the (microlocal) $V$-filtration and
the Bernstein-Sato polynomial of $f$ and its roots; cf. \S\ref{general_Vfiltration}.  Most importantly for us here, and by analogy with the 
description of the log canonical threshold in terms of $V^\bullet \shO_X$, one has
$$\tilde{\alpha}_f = {\rm max}~\{\gamma \in \QQ~|~ \widetilde{V}^\gamma \shO_X = \shO_X\},$$
see for instance \cite[(1.3.8)]{Saito-MLCT}.
Therefore Theorem \ref{general_description} immediately implies the following formula for the log canonicity index of $D$, 
obtained first in \cite{Saito-MLCT} when $\alpha = 1$; recall 
that we are assuming $D = \alpha Z$ with $0 < \alpha \le 1$.

\begin{corollary}\label{mlc}
Let
$$p_0 := {\rm min}~\{p ~|~ I_p (D) \neq \shO_X\} =  {\rm max}~\{p ~|~ (X, D) {\rm~is~}(p-1){\rm -log~canonical}\}.$$
Then $p_0 = [\tilde{\alpha}_f - \alpha + 1]$.
\end{corollary}

Corollary \ref{mlc} can be combined with the information in Corollary \ref{triviality_criterion}, coming from the birational description of Hodge ideals, in order to obtain the inequality
$$\gamma < [\tilde{\alpha}_f - \alpha] + \alpha + 1, \,\,\,\,\,\,{\rm for~ all} \,\,\,\, 0 < \alpha \le 1.$$
Going back to Question \ref{Lichtin} and the subsequent comments, optimizing as $\alpha$ varies we obtain the following partial positive answer to Lichtin's question:

\begin{theorem}\cite{MP4}\label{MLCT_bound}
We have $\gamma \le \tilde{\alpha}_f$.
\end{theorem}

One can analogously define a local version of the Bernstein-Sato polynomial of $f$ around each point $x \in Z$, and in particular a local version of $\tilde{\alpha}_f $ as well, 
denoted $\tilde{\alpha}_{f, x}$. 
Combining the results above with various statements from \cite{MP2}, \cite{MP3}, again using in a crucial way Hodge ideals for $\QQ$-divisors, one also obtains general 
properties of the invariant $\tilde{\alpha}_f$ that extend important features of the log canonical threshold.                                                                                  

\begin{theorem}\cite{MP4}
Let $n = \dim X$.  Then: 
\begin{enumerate}
\item If $x\in X$, $m={\rm mult}_x Z\geq 2$, and $\dim {\rm Sing} \big(\PP(C_xD)\big) = r$,\footnote{We use the convention that $r = -1$ when $\PP(C_xD)$ is smooth.}
where  $\PP(C_xD)$ is the projectivized tangent cone at $x$, then
$$\frac{n- r -1}{m}\le \widetilde{\alpha}_{f,x}\le \frac{n}{m}.$$
\item If $Y$ is a smooth irreducible hypersurface in $X$ such that $Y\not\subseteq Z$, then for every
$x\in Z\cap Y$, we have
$$\widetilde{\alpha}_{f\vert_Y,x}\leq \widetilde{\alpha}_{f,x}.$$
\item Let  $\pi\colon X\to T$ be a smooth morphism, and $s\colon T\to X$ a section such that 
$s(T)\subseteq Z$.  If $Z$ does not contain any fiber of $\pi$,  then the function $T \to \QQ$ given by 
$$t\to \widetilde{\alpha}_{f|_{X_t},s(t)}$$
is lower semicontinuous.
\end{enumerate}
\end{theorem}

It follows for instance from (1) that for every singular point $x \in Z$ we have $\tilde{\alpha}_{f, x} \le n/2$.
This bound is optimal, since by Example \ref{homogeneous} for a quadric $f = x_1^2 + \cdots + x_n^2$ one has $\tilde{\alpha}_f = n/2$.
(In fact  we see in (1) that for every ordinary singular point, i.e  where $ r = -1$, we have $\tilde{\alpha}_f = n/m$, a well-known fact.) 
Saito shows in fact in \cite[Theorem~0.4]{Saito-MV} that all the negatives of the roots of $\tilde{b}_{f,x} (s)$ 
are contained in the interval $[\tilde{\alpha}_{f,x}, n - \tilde{\alpha}_{f,x}]$. Via Theorem \ref{saito}, this is related to Corollary \ref{smoothness_criterion}.

\bibliographystyle{amsalpha}
\bibliography{bibliography}

\end{document}